\documentclass{amsart}
\usepackage{verbatim,amssymb,amsmath,amscd,latexsym,amsbsy,mathrsfs} \input{xy}
\xyoption{all}

\newtheorem{thm}{Theorem}[section] \newtheorem{pro}[thm]{Proposition}
\newtheorem{lm}[thm]{Lemma}
\newtheorem{cor}[thm]{Corollary}
\numberwithin{equation}{section}

\theoremstyle{remark}

\theoremstyle{definition} \newtheorem*{ex}{Example}
\newtheorem{rmk}[thm]{Remark} 
\newtheorem*{note}{Note}
\newtheorem{example}[thm]{Example}
\newtheorem{df}[thm]{Definition}

\DeclareMathAlphabet{\mathpzc}{OT1}{pzc}{m}{it}
\DeclareMathOperator*{\spec}{Spec} \DeclareMathOperator*{\Hom}{Hom}
\DeclareMathOperator*{\card}{card} \DeclareMathOperator*{\im}{Im}
\DeclareMathOperator*{\Aut}{Aut} 
\DeclareMathOperator*{\Gal}{Gal} 
 
\DeclareMathOperator*{\rank}{rank}
\DeclareMathOperator*{\Char}{char}

\DeclareMathOperator*{\into}{\hookrightarrow}
 
 \newcommand{\QQ}{\mathbb{Q}}
\newcommand{\ZZ}{\mathbb{Z}} \newcommand{\Aff}{\mathbb{A}}
\newcommand{\PP}{\mathbb{P}} \newcommand{\FF}{\mathbb{F}}
 
 \newcommand{\cF}{\mathcal{F}}
 \newcommand{\cC}{\mathcal{C}}
 
\newcommand{\cO}{\mathcal{O}}

\newcommand{\invlim} {\displaystyle \lim_{\longleftarrow}}

\newcommand{\onto}{\twoheadrightarrow}

\begin{document}
\title{The fundamental group of affine curves in positive characteristic} 
  \author{
  Manish Kumar
  }
  \address{
Department of Mathematics \\
Universit\"at Duisburg-Essen \\
45117 Essen, Germany
  }
  \email{manish.kumar@uni-due.de}
\begin{abstract}
  It is shown that the commutator subgroup of the fundamental
  group of a smooth irreducible affine curve over an uncountable algebraically 
  closed field $k$ of positive characteristic is a profinite free group of rank 
  equal to the cardinality of $k$.
\end{abstract}
\date{} \maketitle

\section{Introduction}
The algebraic (\'etale) fundamental group of an affine curve over an algebraically closed
field $k$ of positive characteristic has a complicated structure. It is an 
infinitely generated profinite group, in fact the rank of this group is same as the 
cardinality of $k$.
The situation when $k$ is of characteristic zero is simpler to understand.
The fundamental group of a smooth curve over an algebraically closed 
field of characteristic zero
is just the profinite completion of the topological fundamental group (\cite[XIII, 
Corollary 2.12, page 392]{SGA1}).
In positive characteristic as well, Grothendieck gave a description of the prime-to-$p$ 
quotient of the fundamental group of a smooth curve which
is in fact analogous to the characteristic zero case. 
From now on, we shall assume that the characteristic of the base field $k$
is $p>0$. Consider the following exact sequence for the
fundamental group of a smooth affine curve $C$.
\begin{equation*}
1\rightarrow \pi_1 ^c (C) \rightarrow \pi_1(C) \rightarrow \pi_1
^{ab}(C) \rightarrow 1
\end{equation*}
where $\pi_1 ^c(C)$ and $\pi_1 ^{ab}(C)$ are
the commutator subgroup and the abelianization of the fundamental
group $\pi_1(C)$ of $C$ respectively. In \cite{kum}, a description of 
$\pi_1^{ab}(C)$ was given \cite[Corollary 3.5]{kum} and it was also shown that 
$\pi_1^c(C)$ is a free profinite group of countable rank 
if $k$ is countable \cite[Theorem 1.2]{kum}. In fact some more exact sequences with
free profinite kernel like the above were also observed \cite[Theorem 7.1]{kum}. Later
using somewhat similar ideas and some profinite group theory Pacheco, Stevenson and
Zalesskii claimed to find a condition for a closed normal subgroup of $\pi_1(C)$ to be 
profinite free of countable rank \cite{PSZ} but unfortunately 
there seems to be a gap in their argument as Example \ref{nonfree-ex} suggests. 

A consequence of the main result of this paper generalizes \cite[Theorem 1.2]{kum} to uncountable
fields.
\begin{thm} \label{main-thm}
Let $C$ be a smooth affine curve over an algebraically closed field $k$ 
(possibly uncountable) of characteristic $p$ then $\pi_1^c(C)$
is a free profinite group of rank $\card(k)$.
\end{thm}

Let $P_g(C)$ be the intersection of all index $p$ normal subgroups of 
$\pi_1(C)$ corresponding to \'etale covers of $C$ of genus at least $g$ 
(see Definition \ref{defn:P_g}).
In the main theorem (Theorem \ref{groupisfree}) it is shown that if $\Pi$ 
is a closed normal subgroup of $\pi_1(C)$ of rank $\card(k)$ such that 
$\pi_1(C)/\Pi$ is abelian, $\Pi\subset P_g(C)$ for some $g\ge 0$ and 
for every finite simple group $S$ there exist a surjection from $\Pi$ 
to $\card(k)$ copies of $S$ then $\Pi$ is profinite free.

As a consequence we get the following result.
\begin{cor} \label{othernormalsubgroup}
 Let $\Pi$ be a closed normal subgroup of $P_g(C)$ for some $g\ge 0$. If rank 
 of $P_g(C)/\Pi$ is strictly less than $\card(k)$ then $\Pi$ is profinite free 
 of rank $\card(k)$.
\end{cor}

\begin{proof}
 Note that $P_g(C)$ is a profinite free group of rank $\card(k)$ by Corollary 
 \ref{cor:freenessof-comm-P_g}. 
 So the corollary follows from Melnikov's result on freeness of a normal 
 subgroup of a profinite free group (\cite[Theorem 8.9.4]{RZ}).
\end{proof}

The existence of wildly ramified covers is the primary reason for the algebraic 
fundamental group of an affine curve to be incomprehensible. More precisely, 
there are many $p$-cyclic Artin-Schreier \'etale covers of the affine line. 
In fact there is a positive dimensional configuration space of such covers of 
any fixed genus (see \cite{pries}).
This also suggests that the fundamental group of an affine curve in positive 
characteristic contains much more information
about the curve than in characteristic zero case. In fact Harbater 
and Tamagawa have conjectured that the fundamental group of a
smooth affine curve over an algebraically closed field of characteristic $p$ should
determine the curve completely (as a scheme) and in particular one should be able to 
recover the base field. Harbater and Tamagawa have shown some positive results supporting the 
conjecture. See \cite[Section 3.4]{kum-thesis}, \cite{ha8}, \cite{ta1} and \cite{ta2} for 
more details. 

The above theorem on the commutator subgroup can also be interpreted as an 
analogue of the Shafarevich's conjecture for global
fields. Recall that the Shafarevich conjecture says that the
commutator subgroup of the absolute Galois group of the rational numbers
$\QQ$ is a profinite free group of countable rank. David Harbater \cite{Ha6}, 
Florian Pop \cite{pop} and later 
Dan Haran and Moshe Jarden \cite{HJ} have shown, using
different patching methods, that the absolute Galois group of the function
field of a curve over an algebraically closed field is profinite free of the
rank same as the cardinality of the base field. 
See \cite{Ha7} for more details on these kind of results and questions.

Though the whole fundamental group is not well understood,
a necessary and sufficient condition for a finite group to 
be a quotient of the fundamental group of a smooth affine 
curve was conjectured by Abhyankar. It was proved by
Raynaud for the affine line \cite{Ra1} and by Harbater in general 
\cite{Ha2}. 
\begin{thm}(Harbater, Raynaud)\label{abhconj}
Let $C$ be a smooth affine curve of genus $g$ over an algebraically 
closed field of characteristic $p$. Let $D$ be the smooth compactification of 
$C$ and $\card(D\setminus C)=n+1$. Then a finite group $G$ is a quotient of 
$\pi_1(C)$ if and only if $G/p(G)$ is generated by $2g+n$ elements, where $p(G)$
is the quasi-$p$ subgroup of $G$.
\end{thm}

Section \ref{profinite-group} consists of definitions and results on profinite groups.
This section also reduces Theorem \ref{main-thm} to solving certain embedding 
problems. The last section consists of solutions to these embedding problems. Thanks are due
to David Harbater for some useful discussions regarding the uncountable case. 
I am also grateful to a referee for pointing out errors and gaps in the previous draft.
Lior Bary-Soroker encouraged me to include Corollary \ref{othernormalsubgroup} for which I am thankful to him. I would also like to thank Hilaf Hasson for pointing out some typographical errors.

\section{Profinite group theory}\label{profinite-group}

Notation and contents of this section are inspired from \cite{RZ} and \cite{FJ}.
For a finite group $G$ and a prime number $p$,
let $p(G)$ denote the subgroup of $G$ generated by all the $p$-Sylow
subgroups of $G$. The subgroup $p(G)$ is called the \emph{quasi-$p$} subgroup of $G$. If $G=p(G)$ then $G$ is called a quasi-$p$ group.

A family of finite groups $\cC$ is said to be \emph{almost full} if it satisfies
the following conditions:
\begin{enumerate}
\item A nontrivial group is in $\cC$.
\item If $G$ is in $\cC$ then every subgroup of $G$ is in $\cC$.
\item If $G$ is in $\cC$ then every homomorphic image of $G$ is in $\cC$.
\item If $G_1, G_2, \ldots,G_n$ are in $\cC$ then the product $G_1\times G_2\times\ldots\times G_n$ is in $\cC$.
\end{enumerate}
Moreover $\cC$ is called a \emph{full} family if it is closed under extensions,
i.e., if $G_1$ and $G_3$ are in $\cC$ and there is a short exact sequence 
$$1\to G_1\to G_2 \to G_3\to 1$$
then $G_2$ is in $\cC$.

\begin{ex}
The family of all finite groups is full. For a prime number $p$, the family of
all $p$-groups is full. 
\end{ex}

Let $\cC$ be an almost full family of finite groups. A \emph{pro-$\cC$} group is a 
profinite group whose finite quotients lie in $\cC$. Equivalently, it is an inverse
limit of an inverse system of groups contained in $\cC$.  If $\cC$ is the family
of all $p$-groups then pro-$\cC$ groups are also called pro-$p$ groups. 

Let $m$ be an infinite cardinal or a positive integer.
A subset $I$ of a profinite group $\Pi$ is called a \emph{generating set} if
the smallest closed subgroup of $\Pi$ containing $I$ is $\Pi$ itself. A generating
set $I$ is said to be converging to 1 if every open normal subgroup of $\Pi$
contains all but finitely many elements of $I$. The \emph{rank} of $\Pi$ is the 
infimum of the cardinalities of all the generating sets of $\Pi$ converging to 1.

A profinite group $\Pi$ is called a \emph{free} pro-$\cC$ group of rank $m$ if
$\Pi$ is isomorphic to the inverse limit of the inverse system obtained by taking quotients 
of $F_I$ by open normal subgroups $K$ which contain all but finitely many 
elements of $I$ and $F_I/K \in \cC$. The image of $I$ under the natural map 
$F_I \to \hat{F}_I=\Pi$ is a generating set converging to 1.
When $\cC$ is the family of all finite groups then
free pro-$\cC$ groups are same as free profinite groups.

For a group $\Pi$ and a finite group $S$, let $R_S(\Pi)$ denote the cardinality 
of the maximal cardinal $m$ such that there exist a surjection from $\Pi$ to 
the product of $m$ copies of $S$. The intersection of all the proper normal maximal subgroups 
of $\Pi$ is denoted by $M(\Pi)$.

An \emph{embedding problem} consists of surjections $\phi\colon \Pi \onto G$
and $\alpha\colon \Gamma \onto G$        
\begin{equation*}
\xymatrix{
  &          &                        &\Pi \ar @{-->}[dl]_{\psi} \ar[d]^{\phi}\\
  1\ar[r] & H \ar[r] & \Gamma \ar[r]_{\alpha} & G \ar[r] \ar[d] & 1\\
  & & & 1 }
\end{equation*}
where $G$, $\Gamma$ and $\Pi$ are groups and $H=\ker(\alpha)$.
It is also sometimes called an embedding problem for $\Pi$.
It is said to have a \emph{weak solution} if there exists a group
homomorphism $\psi$ which makes the diagram commutative, i.e., $\alpha
\circ \psi =\phi$. Moreover, if $\psi$ is an epimorphism then it is
said to have a \emph{proper solution} (or a \emph{solution}). It is said to be a finite
embedding problem if $\Gamma$ is finite. An embedding problem is said to be a
\emph{split} if there exists a group homomorphism from $G$
to $\Gamma$ which is a right inverse of $\alpha$. Two proper solutions $\psi_1$ and
$\psi_2$ are said to be distinct if $\ker(\psi_1)\ne \ker(\psi_2)$.

Let $\cC$ be an almost full family. A profinite group $\Pi$ of rank $m$ is called 
\emph{$\cC$-homogeneous} if every embedding problem of the following type has a solution:
  $$
  \xymatrix{
    &          &               &\Pi \ar[d] \ar @{-->} [dl]\\
    1\ar[r] & \tilde{H} \ar[r] & \tilde{\Gamma} \ar[r] & \tilde{G} \ar[r]\ar[d] & 1\\
    & & & 1 }
  $$
Here $\tilde{\Gamma}$ and $\tilde{G}$ can be any pro-$\cC$ groups with 
$\rank(\tilde{\Gamma})\le m$ and $\rank(\tilde{G})<m$. Moreover, 
$\tilde{H} \in \cC$ is a minimal normal subgroup of $\tilde{\Gamma}$ 
and is contained in $M(\tilde{\Gamma})$.
If $\cC$ is the class of all finite groups then $\Pi$ is called \emph{homogeneous}.

\begin{note}
In view of \cite[Lemma 3.5.4]{RZ}, the above definition of homogeneous is equivalent to the
definitions given in \cite{RZ} and \cite{FJ}.
\end{note}
Let $m$ denote an infinite cardinal and $\cC$ be an almost full family.
The following is an easy generalization of \cite[Lemma 25.1.5]{FJ}. 
The proof is exactly the same but is reproduced here for the sake of completion.

\begin{lm} \label{FJ-lemma}
Let $\Pi$ be a profinite group such that every nontrivial finite embedding problem 
$\phi:\Pi \onto G, \alpha:\Gamma \onto G$
with $\Gamma$ in $\cC$ and $H=\ker(\alpha)$ a minimal normal subgroup of $\Gamma$ has 
$m$ solutions. Then the following embedding problem has a solution.
\begin{equation}\label{EP-FJ}
  \xymatrix{
    &          &               &\Pi \ar[d] \ar @{-->} [dl]\\
    1\ar[r] & \tilde{H} \ar[r] & \tilde{\Gamma} \ar[r] & \tilde{G} \ar[r]\ar[d] & 1\\
    & & & 1 }
\end{equation}
Here $\tilde{\Gamma}$ and $\tilde{G}$ are pro-$\cC$ groups with 
$\rank(\tilde \Gamma)\le m$ and $\rank(\tilde G)<m$. And $\tilde{H} \in \cC$ is a 
minimal normal subgroup of $\tilde{\Gamma}$.

Moreover, if the existence of solutions to only those embedding problems with $\tilde{H} \subset 
M(\tilde{\Gamma})$ is desired then the hypothesis can be weakened to the existence 
of $m$ solutions to finite embedding problems in which $H=\ker(\alpha)$
is contained in $M(\Gamma)$.
\end{lm} 

\begin{proof}
Consider the embedding problem \eqref{EP-FJ}.
Since $\tilde{H}$ is finite, there exist an open normal subgroup $N$ of 
$\tilde{\Gamma}$ such that $N \cap \tilde{H}= \{1\}$. Taking quotient by $N$, we 
get a finite embedding problem
  $$
  \xymatrix{
    &          &               &\Pi \ar[d] \ar @{-->} [dl]\\
    1\ar[r] & H \ar[r] & \Gamma \ar[r]^{\alpha_0} & G \ar[r]\ar[d] & 1\\
    & & & 1 }
  $$
where $\Gamma=\tilde{\Gamma}/N$, $G=\tilde{G}/\alpha{N}$, the subgroup $H$ of
$\Gamma$ is isomorphic to $\tilde{H}$ and $\Gamma$ is in $\cC$. Moreover if 
we assume that $\tilde{H} \subset M(\tilde{\Gamma})$ then $H \subset M(\Gamma)$.
The reason being, every maximal normal subgroup of $\Gamma$ is a quotient of a 
maximal normal subgroup of $\tilde{\Gamma}$ containing $N$.
The rest of the proof is same as that of \cite[Lemma 25.1.5]{FJ}.
We have the following scenario:
\begin{equation*}
\xymatrix{
       \Pi \ar@{->>}[drr]^{\phi} \ar@{-->>}[ddr] \\
\tilde{H}\ar@{^{(}->}[r]\ar@{=}[d]& \tilde{\Gamma}\ar@{->>}[r]^{\alpha}\ar@{->>}[d]& \tilde{G}\ar@{->>}[d]\\
H\ar@{^{(}->}[r]& \Gamma\ar@{->>}[r]^{\alpha_0}& G 
}
\end{equation*}
By assumption there exist $\beta:\Pi \onto \Gamma$ which makes the above diagram 
commutative. In fact there are $m$ choices for $\beta$.
If $\ker(\phi)\subset \ker(\beta)$ then $\beta$ factors through $\tilde{G}$.
By \cite[Lemma 25.1.1]{FJ}, there are at most $\rank(\tilde{G})<m$ surjections from
$\tilde{G}$ to $\Gamma$. Hence we can choose $\beta$ so that $\ker{\beta}$ does not
contain $\ker{\phi}$. Since $\tilde{\Gamma}$ is the fiber product of $\tilde{G}$
and $\Gamma$ over $G$, the maps $\beta$ and $\phi$ induce a map $\gamma:\Pi \to 
\tilde{\Gamma}$ so that the following diagram commutes:
\begin{equation*}
\xymatrix{
    \Pi \ar@/^/@{->>}[drr]^{\phi} \ar@/_/@{->>}[ddr]_{\beta} \ar[rd]|-{\gamma} \\
& \tilde{\Gamma}\ar@{->>}[r]^{\alpha}\ar@{->>}[d]& \tilde{G}\ar@{->>}[d]\\
& \Gamma\ar@{->>}[r]^{\alpha_0}& G 
}
\end{equation*}

By \cite[Lemma 24.4.1]{FJ} there exist a group $G'$ which fits in the following 
diagram:
\begin{equation*}
\xymatrix{
\im(\gamma)\ar@{->>}[r]^{\alpha}\ar@{->>}[d]& \tilde{G}\ar@{->>}[d]_{\zeta'}\ar@/^/@{->>}[ddr]\\
\Gamma\ar@{->>}[r]^{\alpha_0'}\ar@/_/@{->>}[rrd]_{\alpha_0}& G' \ar@{->>}[rd]|-{\theta}\\
& & G
}
\end{equation*}
the maps from $\im(\gamma)$ are the restriction of maps from $\tilde{\Gamma}$ 
and $\im(\gamma)$ is the fiber product of $\tilde{G}$ and $\Gamma$ over $G'$.
Since $H=\ker(\alpha_0)$ is a minimal normal subgroup of $\Gamma$, one of $\theta$
or $\alpha_0'$ is an isomorphism. If $\alpha_0'$ where an isomorphism then
$\beta=\alpha_0'^{-1}\circ\zeta'\circ\phi$ contradicting $\ker(\phi)$ is
not a subset of $\ker(\beta)$. Hence $\theta$ is an isomorphism. So 
again by \cite[Lemma 24.4.1]{FJ}, $\im(\gamma)=\tilde{\Gamma}$ solves the 
embedding problem \eqref{EP-FJ}.
\end{proof}

\begin{thm}\label{profinite}
Let $\Pi$ be a profinite group of rank $m$. Suppose: 
\begin{enumerate}
  \item $\Pi$ is projective.
  \item
  Every nontrivial finite embedding problem
\begin{equation}\label{EP0}
\xymatrix{
    &          &               &\Pi \ar[d] \ar @{-->} [dl]\\
    1\ar[r] & H \ar[r] & \Gamma \ar[r] & G \ar[r]\ar[d] & 1\\
    & & & 1 }
\end{equation}
  with $H$ a quasi-p group, minimal normal subgroup of $\Gamma$ and $H\subset M(\Gamma)$ has $m$ 
  solutions.
\item Every nontrivial finite split embedding problem \eqref{EP0} with $H$ prime-to-$p$ group and minimal
normal subgroup of $\Gamma$ has $m$ solutions.
\end{enumerate}
Then $\Pi$ is homogeneous.  
Moreover, if $R_S(\Pi)=m$ for every finite simple group $S$ then $\Pi$ is a profinite
free group.
\end{thm}

\begin{proof}
First of all, let us observe that (1), (2) and (3) allow us to assume that every finite
nontrivial (not necessarily split) embedding problem \eqref{EP0} with $H\subset M(\Gamma)$ 
has $m$ solutions. The proof is via induction on $|H|$. 

Note that $p(H)$ is a normal subgroup of $\Gamma$. 
Since $H$ is a minimal normal subgroup of $\Gamma$, either $p(H)=H$ or $p(H)$ is trivial. 
If $p(H)=H$ then (2) guarantees $m$ solutions to the embedding problem \eqref{EP0}
 
If $p(H)$ is trivial then $H$ is a prime-to-$p$ group. The embedding problem
\eqref{EP0} has a weak solution $\phi$ since $\Pi$ is projective. Let $G' \le \Gamma$
be the image of $\phi$. The subgroup $G'$ acts on $H$ via conjugation so we can define 
$\Gamma'=H\rtimes G'$ and get the following embedding problem:
\begin{equation}\label{EP1}
\xymatrix{
    &          &               &\Pi \ar[d]^{\phi} \ar @{-->} [dl]\\
    1\ar[r] & H \ar[r] & \Gamma' \ar[r] & G' \ar[r]\ar[d] & 1\\
    & & & 1 }
\end{equation}

Also $\Gamma'$ surjects onto $\Gamma$ under the homomorphism sending
$(h,g) \mapsto hg$. So it is enough to find $m$ solutions to the embedding
problem \eqref{EP1}.
Now if $H$ is not a minimal normal subgroup of $\Gamma'$ then there exist $H'$
proper nontrivial subgroup of $H$ and normal in $\Gamma'$. Quotienting by $H'$ we
get the following embedding problem:
  $$
  \xymatrix{
    &          &               &\Pi \ar[d]^{\phi} \ar @{-->} [dl]\\
    1\ar[r] & H/H' \ar[r] & \Gamma'/H' \ar[r] & G' \ar[r]\ar[d] & 1\\
    & & & 1 }
  $$
 which has $m$ solutions by induction hypothesis (since $|H/H'|<|H|$). 
 For each solution $\theta'$ to the above, the following embedding 
 problem:
   $$
  \xymatrix{
    &          &               &\Pi \ar[d]^{\theta'} \ar @{-->} [dl]\\
    1\ar[r] & H' \ar[r] & \Gamma' \ar[r] & \Gamma'/H' \ar[r]\ar[d] & 1\\
    & & & 1 }
  $$
also has $m$ solutions by induction hypothesis as $|H'|<|H|$. Let $\theta$ be 
solution to this embedding problem then it  
is in fact a solution to \eqref{EP1} as well. Note that distinct solutions for \eqref{EP1}
induce distinct solutions for \eqref{EP0} by \cite[Lemma 2.4]{HS}.
Finally if $H$ is a minimal normal 
subgroup of $\Gamma'$ then hypothesis (3) guarantees $m$ solutions to \eqref{EP1}.
Lemma \ref{FJ-lemma} yields $\Pi$ is homogeneous.
The rest of the statement follows from \cite[Theorem 8.5.2]{RZ} and 
\cite[Lemma 3.5.4]{RZ}. 
\end{proof}

Let $\Gamma$ be a finite group, $H$ a normal 
subgroup of $\Gamma$ contained in $M(\Gamma)$, $G=\Gamma/H$ and $\alpha:\Gamma\to G$ be 
the quotient map. Let $\Pi$ be a closed normal subgroup of 
a profinite group $\Theta$.
\begin{lm}\label{restriction-lemma}
Suppose we have a surjection $\psi$ from $\Theta \to G$
which restricted to $\Pi$ is also a surjection.
$$
\xymatrix{
  & \Pi \ar@{^{(}->}[r]\ar@{->>}[rd] & \Theta \ar@{->>}[d]^{\psi}\ar@{->>}[dl]\\
  H\ar@{^{(}->}[r] &\Gamma\ar@{->>}[r]_{\alpha} &G
}
$$ 
Let $\phi$ be a surjection of $\Theta$ onto $\Gamma$ such that 
$\psi=\alpha\circ\phi$.
Then the restriction of $\phi$ to $\Pi$ is a surjection 
onto $\Gamma$.
\end{lm}

\begin{proof}
Let us first note that $\phi(\Pi)$ is  a normal subgroup of $\Gamma$ and 
$\alpha(\phi(\Pi))=G$. Suppose $\phi(\Pi)$ is a proper normal subgroup of $\Gamma$,
then there exists a maximal normal proper subgroup $\Gamma'$ of $\Gamma$ 
containing $\phi(\Pi)$. Since  $H$ is contained in $ M(\Gamma)$, 
$H\subset \Gamma'$.
Also $\alpha(\Gamma')=G$, so $\Gamma'=\Gamma$ contradicting that $\Gamma'$ is 
a proper subgroup of $\Gamma$. Hence $\phi(\Pi)=\Gamma$.

\end{proof}

\section{Solutions to embedding problems} \label{solutionsEP}

A morphism of schemes, $\Phi\colon X\rightarrow Y$, is said to
be a \emph{cover} if $\Phi$ is finite, surjective and generically separable.
For a finite group $G$, $\Phi$ is said to be a \emph{$G$-cover} 
(or a $G$-Galois cover) if in
addition there exists a group monomorphism $G\rightarrow \Aut_Y(X)$
which acts transitively on the geometric generic fibers of $\Phi$.
Let $k_0\subset k$ be fields. For a $k_0$-scheme $X$, 
$X\times_{\spec(k_0)}\spec(k)$ will also be denoted by $X\otimes_{k_0} k$.
\begin{df} \label{df:fieldofdefn}
Let $X$ and $Y$ be varieties over $k$ and a $k$-morphism $f:X\to Y$ be a cover. 
We shall say that \emph{$f$ is defined over $k_0$} if there exist a cover 
$f_0:X_0\to Y_0$ of $k_0$-varieties which is a $k_0$-morphism and a 
cover $g:Y\to Y_0\otimes_{k_0} k$ which is a $k$-morphism 
such that the normalized pull-back of $f_0\times id_{k}:X_0\otimes_{k_0}k\to Y_0\otimes_{k_0}k$ along $g$ restricted to a dominating component $Z$ of the normalization of $X_0\times_{Y_0} Y$ factors through $f$, i.e.,
\[
\xymatrix{ 
   & X\ar[rd]^f \\
  Z \ar[rr]\ar[d]\ar[ru] & & Y\ar[d]^g\\
  X_0\otimes_{k_0} k\ar[rr]^{f_0\times id_k} & & Y_0\otimes_{k_0} k
}
\]
\end{df}
Note that $k(Y)$ and $k(X_0)$ can be viewed as subfields of some fixed algebraic closure of $k(Y_0)$ and $k(Z)$ is the compositum $k(X_0)k(Y)$. So $f$ is defined over $k_0$ implies that $k(X)$ is contained in the compositum $k(X_0)k(Y)$.

Let $Y$ be an irreducible curve over $k$. We will say that a \emph{finite field extension $L/k(Y)$ is defined over $k_0$} if the natural morphism ${\bar Y}^L\to Y$ is defined over $k_0$, where ${\bar Y}^L$ is the normalization of $Y$ in $L$.

\begin{rmk}\label{basicproperty}
 Let $W\to X$ and $X\to Y$ be covers of $k$-varieties. If the composition $W\to Y$ is defined over $k_0$ then the covers $W\to X$ and $X\to Y$ are defined over $k_0$.
\end{rmk}

\begin{lm}\label{lm:fieldofdefn_extn}
 Let $k_0\subset k$ be algebraically closed fields. Let $f:X\to Y$ be a cover of smooth proper $k$-curves, $f_0:X_0\to Y_0$ be a cover of smooth proper $k_0$-curves and $g:Y\to Y_0\otimes_{k_0} k$ be a cover. If $k(X_0)k(Y)=k(X)$ then $X\to Y$ is defined over $k_0$.
\end{lm}

\begin{proof}
 Since $k(X_0)\subset k(X)$, we have a dominant rational map from $X\to X_0\otimes_{k_0}k$. Moreover since $X_0$ and $X$ are smooth and proper curves, the rational map $X\to X_0\otimes_{k_0}k$ is a proper surjective morphism. So we have the following commutative diagram.
 \[
\xymatrix{ 
  X \ar[r]^f \ar[d]  & Y\ar[d]^g\\
  X_0\otimes_{k_0} k\ar[r]^{f_0} & Y_0\otimes_{k_0} k
}
\]
 Finally, since $k(X)$ is the compositum of $k(X_0)$ and $k(Y)$, $X$ is the normalization of a dominating component of $X_0\times_{Y_0} Y$.
\end{proof}

\begin{lm}\label{reductionforgaloiscover}
 Let $k_0\subset k$ be algebraically closed fields and $G$ a finite group. Let 
 $f:X\to Y$ be a $G$-Galois cover of smooth $k$-curves. If $f$ is defined over $k_0$ then the covers $f_0:X_0\to Y_0$ and $g:Y\to Y_0\otimes_{k_0} k$ in Definition \ref{df:fieldofdefn} above can be chosen so that $f_0$ is a $G$-Galois cover and $X$ is the normalization of $X_0\times_{Y_0} Y$.   
\end{lm}

\begin{proof}
 First we observe that $f_0:X_0\to Y_0$ in Definition 3.1 is a cover of $k_0$-curves. By taking Galois closure $L$ of the field extension $k_0(X_0)/k_0(Y_0)$, and replacing $X_0$ by the normalization of $X_0$ in $L$ we may assume $f_0$ is a Galois cover and $X_0$ is normal. By assumption, there is a dominating component $Z$ of the normalization of $X_0\times_{Y_0} Y$ lying above $X$. Note that $Z$ is the normalization of $X$ in $k(Y)k(X_0)$. 
 
 We observe that $k_0(Y_0)\subset k_0(X_0)\cap k(Y)$ and the latter field is a finite extension of $k_0(Y_0)$. Let $Y_1$ be the normalization of $Y_0$ in $k_0(X_0)\cap k(Y)$. We have $k(Y)$ and $k(X_0)$ are linearly disjoint over $k(Y_1)$. Since $Y$ is a proper normal curve, $k(Y)/k(Y_0)$ is a finite separable extension, and $k(Y_0)\subset k(Y_1)\subset k(Y)$, by valuative criterion for properness there exist a cover  $g':Y\to Y_1\otimes_{k_0} k$ such that the morphism $g:Y\to Y_0\otimes_{k_0} k$ factors through $g'$. 

 Similarly, since $X_0$ is a proper normal curve, the Galois cover $f_0:X_0\to Y_0$ factors through a $k_0$-morphism $f'_0:X_0\to Y_1$ which is also a Galois cover of $k_0$-curves. Let $G'=\Gal(k(X_0)/k(Y_1))$. Linear disjointness of $k(X_0)$ and $k(Y)$ over $k(Y_1)$ implies that the Galois group of $k(Z)=k(X_0)k(Y)$ over $k(Y)$ is also $G'$. Since $k(X)/k(Y)$ is also Galois with Galois group $G$ and $k(X)\subset k(Z)$, we have a group epimorphism $G'\onto G$. Let $H$ be its kernel. Then $k(X)= k(Z)^H$, the subfield fixed by $H$.
 Let $X_1$ be the normalization of $Y_1$ in the fixed field $k_0(X_0)^H$, then we obtain a $k_0$-morphism $f_1:X_1\to Y_1$ which is a $G$-cover. Moreover, $H$ viewed as a subgroup of $\Gal(k(Z)/k(Y))=G'$ acts trivially on $k(X_1)k(Y)$. Since $k(X_1)\subset k(X_0)$, $k(X_1)$ and $k(Y)$ are also linearly disjonit over $k(Y_1)$. So the Galois group $\Gal(k(X_1)k(Y)/k(Y))=\Gal(k(X_1)/k(Y_1))=G$. Hence $k(X_1)k(Y)=k(Z)^H=k(X)$. Hence $X$ is a dominating component of the normalization of $X_1\times_{Y_1} Y$. Also $k(X_1)$ and $k(Y)$ are linearly disjoint over $k(Y_1)$, so the normalization of $X_1\times_{Y_1} Y$ is connected. Hence $X$ is the normalization of $X_1\times_{Y_1} Y$. The assertion of the lemma is now obtained by replacing $X_0$, $Y_0$, $f_0$ and $g$ in Definition \ref{df:fieldofdefn} by $X_1$, $Y_1$, $f_1$ and $g'$ respectively.
\end{proof}

\begin{lm}\label{lm:compositum}
 Let $k_0\subset k$ be algebraically closed fields and $Y$ a smooth $k$-curve. Let the finite Galois extensions $L_1/k(Y)$ and $L_2/k(Y)$ be defined over $k_0$. Then the compositum $L_1L_2/k(Y)$ is also defined over $k_0$. Here all the fields are subfields of some fixed algebraic closure of $k(Y)$.
\end{lm}

\begin{proof}
 Let $X_1$, $X_2$ and $X$ be the normalization of $Y$ in $L_1$, $L_2$ and $L_1L_2$  respectively. By assumption the covers $X_1\to Y$ and $X_2\to Y$ are defined over $k_0$. By Lemma \ref{reductionforgaloiscover} there exist covers $X_{10}\to Y_{10}$ and $X_{20}\to Y_{20}$ of $k_0$-curves and covering morphisms $Y\to Y_{10}\otimes_{k_0}k$ and $Y\to Y_{20}\otimes_{k_0}k$ such that $X_1$ is the normalization of $X_{10}\times_{Y_{10}}Y$ and $X_2$ is the normalization of $X_{20}\times_{Y_{20}}Y$. In particular, $k(X_1)=k(X_{10})k(Y)$ and $k(X_2)=k(X_{20})k(Y)$. Let $X_0$ be the normalization of $X_{10}$ in $k_0(X_{10})k_0(X_{20})$. Then $X_0\to Y_{10}$ is a cover of smooth $k_0$-curves. Moreover, $$k(X_0)k(Y)=k(X_{10})k(X_{20})k(Y)=k(X_1)k(X_2)=L_1L_2=k(X)$$
 So by Lemma \ref{lm:fieldofdefn_extn}, $X\to Y$ is defined over $k_0$. Hence the extension $L_1L_2/k(Y)$ is defined over $k_0$.
\end{proof}

\begin{lm}\label{lm:descend}
 Let $k_0\subset k$ be algebraically closed fields. Let $f:X\to Y$ be a $G$-cover of smooth proper $k$-curves. Let $Z_0$ be a smooth proper $k_0$-curve and $Z=Z_0\otimes_{k_0} k$. Let $Y\to Z$ be a cover of $k$-curves and $Z'\to Z_0$ be a cover of $k_0$-curves. Let $X'$ and $Y'$ be the normalization of a dominating component of $X\times_{Z_0} Z'$ and $Y\times_{Z_0} Z'$ respectively. Assume that the cover $X'\to Y'$ induced from $f$ is also a $G$-cover and that it is defined over $k_0$. Then $f:X\to Y$ is also defined over $k_0$.
\end{lm}

\begin{proof}
 Since $X'\to Y'$ is a $G$-cover defined over $k_0$, by Lemma \ref{reductionforgaloiscover}, there exist a $G$-cover of proper $k_0$-curves $X_0\to Y_0$ and a cover $Y'\to Y_0\otimes_{k_0} k$ such that $X'$ is the normalization of $X_0\times_{Y_0} Y'$. Note that $k(X_0)k(Y')=k(X')=k(X)k(Z')$ and $k(Y')=k(Y)k(Z')$. 
 We have the following diagram: 
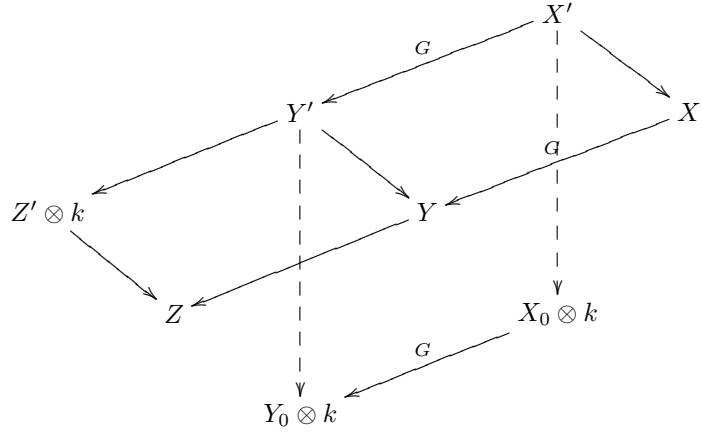
\[\xymatrix{
  & & & & X'\ar[rd]\ar[lld]_G\ar@{-->}[ddd] \\
  & & Y'\ar[rd]\ar[lld]\ar@{-->}[ddd] & & & X\ar[lld]_G\\
  Z'\otimes k\ar[rd]& & & Y\ar[lld] \\
  & Z & & & X_0\otimes k \ar[lld]_G\\
  & & Y_0\otimes k 
  }
 \]
 Let $L=k_0(X_0)k_0(Z')$ be the compositum of fields in $k(X')$. Note that $k_0(Z_0)\subset L$. Let $X_1$ be the normalization of $Z_0$ in $L$. So we have a cover $X_1\to Z_0$ of $k_0$-curves. Also we have a cover $Y\to Z$. Let $X_1'$ be the normalization of a dominating component of $X_1\times_{Z_0} Y$. We claim that the  morphism $X'$ to $X_1'$ induced from $k(X_1')\subset k(X')$ is an isomorphism. Since $X_1'$ and $X'$ are both smooth proper curves, it is enough to verify that $k(X_1')=k(X')$. But this is clear as, $$k(X_1')=k(X_1)k(Y)=k(X_0)k(Z')k(Y)=k(X_0)k(Y')=k(X')$$
 So $X'=X_1'$ dominates $X$, hence $X\to Y$ is defined over $k_0$.
\end{proof}

Next we shall see that if a Galois cover of curves is of degree prime to the characteristic of the base field then there is almost no distinction between the field of definition of the cover and the field of definition of the target curve.

\begin{pro}\label{fieldofdefn_primetop}
 Let $k$ be an algebraically closed field, $C$ a smooth affine $k$-curve and $f:U\to 
 C$ be an \'etale Galois cover of degree prime to $\Char(k)=p$. If $C$ is defined over 
 $k_0\subset k$ then $f$ is defined over the algebraic closure $\bar k_0$ of $k_0$.
\end{pro}

\begin{proof}
 Let $C_{\bar k_0}$ be the $\bar k_0$-curve such that $C=C_{\bar k_0}\otimes_{\bar k_0} k$. Let $H$ be the prime-to-$p$ group such that the $k$-morphism $f:U\to C$ is an \'etale $H$-cover. Let $X_{\bar k_0}$ be the smooth completion of $C_{\bar k_0}$, $g$ be the genus of $X_{\bar k_0}$ and $r$ be the number of points in $X_{\bar k_0}\setminus C_{\bar k_0}$. Note that $X=X_{\bar k_0}\otimes_{\bar k_0}k$ is the smooth completion of $C$, genus of $X$ is $g$ and the number of points in $X\setminus C$ is $r$. Let $\eta_H$ be the number of epimorphisms from the free group on $2g+r-1$ elements $F_{2g+r-1}$ to $H$.

 By Grothendieck's Riemann Existence theorem \cite[XIII, Corollary 2.12]{SGA1} there exist exactly $\eta_H$ distinct $\bar k_0$-morphisms $f_i:Y_i\to X_{\bar k_0}$ which are $H$-covers \'etale over $C_{\bar k_0}$ for $1\le i\le \eta_H$. By the base change to $\spec(k)$ we obtain $\eta_H$ distinct $H$-coverings $f_i\times id_{\spec(k)}:Y_i\otimes_{k_0}k\to X$ which are \'etale over $C$.

 Again by Grothendieck's Riemann Existence theorem there are exactly $\eta_H$ distinct $H$-coverings of $X$ \'etale over $C$. So the given cover $f:U\to C$ is the restriction of $f_i\times id_k: Y_i\otimes_{k_0}k \to X$ to the preimage of $C$ for some $i$ between $1$ and $\eta_H$. Hence the $k$-morphism $f$ is defined over $\bar k_0$.
\end{proof}

The hypothesis that the degree of $f$ is prime-to-$p$ in the above proposition is necessary as the following example of an Artin-Schreier cover shows.
\begin{ex}
 Let $k_0\subset k$ be algebraically closed fields and $\phi:\Aff^1_z \to \Aff^1_x$ be a $\ZZ/p\ZZ$-cover given by $z^p-z-ax$ for some $a$ in $k$. The cover $\phi$ is defined over $k_0$ iff $a\in k_0$. 
\end{ex}

\begin{pro}\label{pro:fieldofdefn}
 Let $k_0 \subset k$ be algebraically closed fields of characteristic $p>0$. Let $X$ be a $k_0$-curve and $C \subset X$ a smooth open $k_0$-curve. Let $\phi:T\to X\otimes_{k_0}k$ be a $\ZZ/p\ZZ$-cover of smooth proper $k$-curves \'etale over $C$ such that $T$ is not defined over $k_0$. Let $\psi: Y\to T$ be an $H$-cover of smooth proper $k$-curves \'etale over $\phi^{-1}(C)$. If $\psi$ is defined over $k_0$ then $\psi$ is the pull-back of an $H$-cover of $X$.
\end{pro}

\begin{proof}
 Since $\psi$ is defined over $k_0$, by Lemma \ref{reductionforgaloiscover} there exists an $H$-cover $\psi_0:Y_0\to T_0$ of $k_0$-curves and a cover $g:T\to T_0\otimes_{k_0}k$ of $k$-curves such that $Y$ is the normalization of $Y_0\times_{T_0} T$. In particular, $k(Y_0)$ and $k(T)$ are linearly disjoint over $k(T_0)$.

 Note that $k_0(T_0)$ and $k_0(X)$ are subfields of $k(T)$ and $k_0(T_0)k_0(X)$ is a field of transcendence degree 1 over $k_0$. Let $T_1$ be the smooth proper $k_0$-curve such that $k_0(T_1)=k_0(T_0)k_0(X)$. Note that $k(X)\subset k(T_1)\subset k(T)$. Since $[k(T):k(X)]=p$, $k(T_1)=k(T)$ or $k(T_1)=k(X)$. As $T_1$ is a $k_0$-curve and $T$ is not defined over $k_0$, $k(T_1)\ne k(T)$. Hence $k(T_1)=k(X)$, so $T_1=X$. But $k_0(T_0)\subset k_0(T_1)=k_0(X)$, so there exist a proper surjective morphism $h:X\to T_0$ of $k_0$-curves. Also we have $k(T_0)\subset k(X)\subset k(T)$ induced from the proper surjective morphisms $T\to X\otimes_{k_0}k$, $T\to T_0\otimes_{k_0} k$ and $X\to T_0$. So the morphism $g:T\to T_0\otimes_{k_0} k$ factors through $X\otimes_{k_0}k$, i.e., we have a proper morphism $\phi:T\to X\otimes_{k_0}k$ such that $(h\otimes_{k_0}k)\circ\phi=g$
 
 Let $Y'$ be the normalization of $Y_0\times_{T_0} X$. Since $k(X)\subset k(T)$, $k(X)$ and $k(Y_0)$ are linearly disjoint over $k(T_0)$. Hence $Y'$ is smooth and connected. After base change to $k$, and looking at the normalized pull-back to $T$, we obtain:
 \[\xymatrix {Y\ar[d]\ar[r]^{\psi} & T\ar[d]^{\phi}\\
Y'\otimes_{k_0}k\ar[d]\ar[r]       & X\otimes_{k_0}k\ar[d]^h\\
Y_0\otimes_{k_0}k\ar[r]_{\psi_0}      & T_0\otimes_{k_0}k }\]
Since $\psi_0$ is an $H$-cover, the morphism $Y'\to X$ is also an $H$-cover.
Since $\psi$ is the pull-back of $Y'\to X$ we are done.
\end{proof}

\begin{cor}
 Let $k_0 \subset k$ be as above. Let  $\phi:T\to \PP^1_x$ be a $\ZZ/p\ZZ$-cover of smooth proper $k$-curves ramified only at $x=\infty$, such that $T$ is not defined over $k_0$. Let $\psi: Y\to T$ be an \'etale $H$-cover of smooth proper $k$-curves where $H$ is a nontrivial prime-to-$p$ group. Then $\psi$ is not defined over $k_0$.
\end{cor}

\begin{proof}
 There are no \'etale $H$-covers of $\Aff^1_x$, so the corollary follows from Proposition \ref{pro:fieldofdefn}.
\end{proof}

Let $k$ be an uncountable algebraically closed field of characteristic $p$.
Let $K^{un}$ denote the compositum (in some fixed algebraic
closure of $k(C)$) of the function fields of all Galois \'etale covers of $C$.
In these notations $\pi_1(C)=\Gal(K^{un}/k(C))$. 
\begin{df}\label{defn:P_g}
For each $g \ge 0$, let
$$P_g(C)=\cap \{\pi_1(Z) : Z\to C \text{ is an \'etale $\ZZ/p\ZZ$-cover and 
genus of } Z \ge g \}$$
be an increasing sequence of closed normal subgroups of $\pi_1(C)$. 
\end{df}

\begin{thm}\label{groupisfree}
Let $\Pi$ be a closed normal subgroup of $\pi_1(C)$ of rank $m$ such that 
$\pi_1(C)/\Pi$ is an abelian group and $\Pi$ is a subset of $P_g(C)$ for 
some $g \ge 0$. Then $\Pi$ is a homogeneous profinite group. 
Moreover if $R_S(\Pi)=m$ for every 
finite simple group $S$ then $\Pi$ is a free profinite group of rank $m$. 
\end{thm}
More precisely, in view of Theorem \ref{profinite}, it will be shown that the
finite embedding problem
  $$
  \xymatrix{
    &          &               &\Pi \ar[d] \ar @{-->} [dl]\\
    1\ar[r] & H \ar[r] & \Gamma \ar[r] & G \ar[r]\ar[d] & 1\\
    & & & 1 }
  $$
has $m$ solutions in the following situations to obtain the above result.
\begin{enumerate}
 \item The kernel $H$ is a quasi-$p$ minimal normal subgroup of $\Gamma$ contained in $M(\Gamma)$ 
  (theorem \ref{quasi-p}).
 \item The embedding problem is split and $H$ is a prime-to-$p$ minimal normal subgroup 
  of $\Gamma$ (theorem \ref{main}). 
\end{enumerate}

Let $K^b$ be the fixed subfield of $K^{un}$ under the action of $\Pi$. So by Galois theory $\Gal(K^{un}/K^b)=\Pi$. Note that the surjection from $\Pi$ to $G$ corresponds 
to a Galois extension $M\subset K^{un}$ of $K^b$ with Galois
group $G$. Since $K^b$ is an algebraic extension of $k(C)$ and $M$
is a finite extension of $K^b$, we can find a finite 
extension $L \subset K^b$ of $k(C)$ and $L'\subset K^{un}$ a
$G$-Galois extension of $L$ so that $M=K^bL'$. Let $\pi_1^L=\Gal(K^{un}/L)$.
The following figure provides a summary.

\begin{figure} [htbp]
\begin{equation*}
  \xymatrix{K^{un}  \\
            &   M\ar@{=}[r]\ar[ul] & L'K^{b} \ar[ull] \\
    &  K^{b}\ar[uul]^{\Pi}\ar[u] \ar[ru]_G & L' \ar[u]\\
    &  L \ar@/^2pc/[uuul]^{\pi_1 ^L} \ar[u] \ar[ru]_G\\
    &  k(C)\ar[u] }
\end{equation*}
\caption{}\label{toweroffields}
\end{figure} 

Let $D$ be the smooth completion of $C$, $X$ be the
normalization of $D$ in $L$ and $ \Phi_X':X\rightarrow D$ be the normalization
morphism. Then $X$ is an abelian cover of $D$ \'etale over
$C$ and its function field $k(X)$ is $L$. Let $W_X$ be
the normalization of $X$ in $L'$ and $\Psi_X$ be the corresponding
normalization morphism. Then $\Psi_X$ is \'etale away from the points lying above
$D\setminus C$ and $k(W_X)=L'$. Since $k$ is algebraically closed, $k(C)/k$ has a 
separating transcendence basis. By a stronger version of Noether normalization 
(for instance, see \cite[Corollary 16.18]{Eis}), there exist a finite generically
separable surjective $k$-morphism from $C$ to $\Aff^1 _x$, where $x$ denotes 
the local coordinate of the affine line. The branch 
locus of such a morphism is codimension $1$, hence this morphism is \'etale away from 
finitely many points. By translation we may assume none of these points maps to $x=0$.
This morphism extends to a finite surjective morphism $\Theta: D\rightarrow \PP^1 _x$.
Let $\Phi_X:X\rightarrow \PP^1 _x$ be the composition $\Theta\circ \Phi_X'$. 
Let $\{r_1,\ldots,r_N\}=\Phi_X^{-1}(\{x=0\})$, then $\Phi_X$
is \'etale at $r_1,\ldots,r_N$. Also note that $\Theta^{-1}(\{x=\infty\})=
D\setminus C$.  

Let $k_0$ be a countable algebraically closed subfield of $k$ such that $X$ and $W_X$ are defined over $k_0$ and the morphisms $\Phi_X$ and $\Psi_X$ are base change of $k_0$-morphism to $k$. We shall denote the corresponding $k_0$-curves and $k_0$-morphisms as well by $X$, $W_X$ and $\Phi_X$ and $\Psi_X$ to simplify notation. Let $\cF$ be a totally ordered set of algebraically closed subfields of $k$ containing $k_0$ under inclusion 
such that for any $k_{\alpha}$ in $\cF$, the transcendence degree of $k_{\alpha}$ over 
$\cup\{k'\in \cF:k'\subsetneq k_{\alpha}\}$ is infinite. 
Note that since $k$ is uncountable, we can find such an $\cF$ with the same 
cardinality as that of $k$.

\subsection{Prime-to-$p$ group}\label{sec:prime-to-p}

Let $\Gamma$ be a finite group, $H$ a prime-to-$p$ nontrivial normal subgroup of $\Gamma$
and $G$ a subgroup of $\Gamma$ such that $\Gamma=H\rtimes G$.
Let $\Phi_Y:Y\to \PP^1_y$ be the smooth $\ZZ/p\ZZ$-cover ramified only at $y=0$ 
given by $z^p-z-y^{-r}$ where $r$ is coprime to $p$ and can be chosen to ensure that
the genus of $Y$ is as large as desired. Recall that $\Pi\subset P_g(C)$ for some $g\ge 0$. The $r$ above is chosen so that
the genus of $Y$ is at least $g$ and greater than the number of generators for $H$. Let $F$ be the locus of $t-xy=0$ in $\PP^1_x\times_k
\PP^1_y\times_k\times_k\spec(k[[t]])$. Let $Y_F=Y\times_{\PP^1_y}F$, where the morphism from 
$F\to \PP^1_y$ is given by the composition of morphisms $F \into \PP^1_x\times_k\PP^1_y\times_k\spec(k[[t]])
 \onto \PP^1_y$. Similarly define $X_F=X\times_{\PP^1_x}F$. 
Let $T$ be the normalization of an irreducible dominating component of
the fiber product $X_F\times_F Y_F$. The situation so far can be 
described by the following picture:

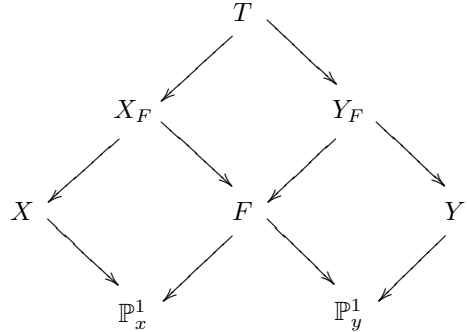
\begin{equation*}
\xymatrix {
        &         & T\ar@<1ex>[dl] \ar@<1ex>[dr]\\
 &X_F\ar@<1ex>[dl]\ar@<1ex>[dr] &  &Y_F\ar@<1ex>[dl]\ar@<1ex>[dr]\\
X\ar@<1ex>[dr]&       &F\ar@<1ex>[dl]\ar@<1ex>[dr]&      &Y\ar@<1ex>[dl] \\
       &\PP^1 _x &             &\PP^1 _y \\
}
\end{equation*}

Let $\Psi_Y:W_Y\to Y$ be an \'etale $H$-cover of $Y$. Note that this is possible 
because $H$ is a prime-to-$p$ group and the genus of $Y$ is greater than the number of generators of $H$
(\cite[XIII, Corollary 2.12, page 392]{SGA1}). Let $T_X$ and $T_Y$ be the open subschemes of $T$ given by $x\ne 0$ and $y\ne 0$ respectively. Let $W_{XF}$ and $W_{XT}$ be the normalized pullback of $W_X\to X$ to $X_F$ and $T_X$ respectively. Similarly define $W_{YF}$ and $W_{YT}$ to be the normalized pullback of $W_Y\to Y$.

\begin{pro}\label{many-covers}
For each $k_{\alpha}\in \cF$, there exist an irreducible $\Gamma$-cover 
$W_{\alpha} \to T_{\alpha}$ of $k_{\alpha}$-curves with the following properties:
\begin{enumerate}
\item $T_{\alpha}$ is a cover of $X$ unramified over the preimage of $C$ in $X$, 
the composition $W_{\alpha} \to T_{\alpha} \to D$ is unramified over $C$ and 
the $G$-cover $W_{\alpha}/H\to T_{\alpha}$ is isomorphic to the cover 
$W_X\times_X T_{\alpha} \to T_{\alpha}$.
\item Let $H'\ne \{e\}$ be a quotient of $H$. If $W_b\to W_X$ is an $H'$-cover 
\'etale over the preimage of $C$ in $W_X$ then $k(W^b)k(T_{\alpha})$ and 
$k(W_{\alpha})$ are linearly disjoint over $k(W_X)k(T_{\alpha})$.
\item The cover $W_{\alpha}\to T_{\alpha}$ is not defined over 
$k'=\cup\{ k_{\beta}\in \cF:k_{\beta} \subsetneq k_{\alpha}\}$.
\item $T_{\alpha}$ is a Galois cover of $D$ with $\Pi\subset\Gal(K^{un}/k(T_{\alpha}))$, 
i.e., $k(T_{\alpha})\subset K^b$.
\end{enumerate}
\end{pro}

\begin{proof}
By \cite[Proposition 6.4]{kum} there exist an irreducible normal $\Gamma$-cover
$W\to T$ of $k[[t]]$-schemes such that over the generic point $\spec(k((t)))$,
$W^g\to T^g$ is ramified only over the points of $T^g$ lying above $x=\infty$.
This cover can be specialized to $k_{\alpha}$ to obtain (1). 
But we shall apply this argument in a slightly modified setup to obtain 
$W_{\alpha}\to T_{\alpha}$ satisfying both (1) and (2). 

For any positive integer $l$, let $\Gamma^{(l)}=H^l\rtimes G$, where the action
of $G$ on $H^l$ is given by component-wise action of $G$ on each copy of $H$.
Increasing $r$ in the definition of $Y$ to increase the genus of $Y$,
we may assume there exist an \'etale $H^l$-cover $W'_Y\to Y$. 
Since $H^l$ is still a prime-to-$p$ group and $\Gamma^{(l)}= H^l\rtimes G$, 
we can apply \cite[Proposition 6.4]{kum} to obtain an 
irreducible normal $\Gamma^{(l)}$-cover $W'\to T$ of $k[[t]]$-schemes such that 
$W'^g\to T^g$ is ramified only over the points of $T^g$ lying above $x=\infty$.

Now this cover can be specialized to obtain covers of $k_{\alpha}$-schemes 
using \cite[Proposition 6.9]{kum}. In fact in the proof of 
\cite[Proposition 6.9]{kum} it was shown that there exist an open subset 
$S$ of the spectrum of a $k[t,t^{-1}]$ algebra such that the coverings $W'^g\to 
W_{XT}^g\to  T^g\to X\otimes_k k((t))$ descends to a cover of $S$-schemes $W'_S\to 
W_{XT,S}\to T_S\to X\times_k S$. Moreover, the fiber over every closed point in 
$S$ leads to a cover of smooth $k$-curves with the desired ramification properties 
and Galois groups same as that over the generic point $\spec(k((t)))$. 
So the fiber over a $k_{\alpha}$-point of $S$ provides a $\Gamma^{(l)}$-cover
$W'_{\alpha}\to T_{\alpha}$ of $k_{\alpha}$-curves. Moreover, let 
$H_i=H\times\ldots H\times\hat H\times H\times\ldots H$ be the subgroup of 
$\Gamma^{(l)}$ where the $i^{\text{th}}$ factor of $H$ is replaced by the trivial 
group. Then the quotients $W'_{\alpha}/H_i\to T_{\alpha}$ are $\Gamma$-covers
satisfying (1). Let $W^i_{\alpha}=W'_{\alpha}/H_i$. 

Since $H$ is a prime-to-$p$ group, so is any quotient $H'$ of $H$. Hence there are 
only finitely many $H'$-covers of $W_X$ which are \'etale over the preimage of $C$ 
and hence there are only finitely many subcovers of these $H'$-covers of $W_X$.
Also, since $H$ is a finite group, it has only finitely many quotients.
We choose $l$ to be greater than the total number of all the subcovers of 
such $H'$-covers for all the quotients $H'$ of $H$. 
Note that for each $i$, $k(W^i_{\alpha})$ is linearly disjoint with the 
compositum of the remaining $k(W^j_{\alpha})$'s over $k(W_X)k(T_{\alpha})$. 
We claim that at least one of these $W^i_{\alpha}$ satisfy (2) as well. Suppose
not, then for each $i$ there exist an $H'$-cover $W^i_b\to W_X$ \'etale over the 
preimage of $C$ such that
$k(W^i_b)k(T_{\alpha})$ and $k(W^i_{\alpha})$ are not linearly disjoint over 
$k(T_{\alpha})k(W_X)$ where $H'\ne \{e\}$ is a quotient of $H$. 
Since $k(T_{\alpha})k(W_X)/k(W_X)$ is $p$-cyclic extension and $k(W^i_b)/k(W_X)$
is a prime-to-$p$ extension, $k(T_{\alpha})k(W_X)$ and $k(W^i_b)$ are linearly 
disjoint over $k(W_X)$. So $k(W^i_b)$ and $k(W^i_{\alpha})$ are not 
linearly disjoint over $k(W_X)$. Hence $M_i:= k(W^i_b)\cap k(W^i_{\alpha})\supsetneq 
k(W_X)$. Note that $M_i$ defines a nontrivial subcover of the $H'$-cover $W^i_b\to 
W_X$ and $k(W_X)k(T_{\alpha})\subsetneq M_ik(T_{\alpha})\subset k(W^i_{\alpha})$. 
But linear disjointness of $k(W^i_{\alpha})$ with the compositum of the 
remaining $k(W^j_{\alpha})$'s over $k(W_X)k(T_{\alpha})$ tells us that 
$M_i\ne M_j$ for $i\ne j$. So we have produced $l$ distinct covers of $W_X$ such that
each one is a subcover of some $H'$-cover of $W_X$ \'etale over the preimage of $C$ for
some $H'$ a quotient of $H$. This contradicts the choice 
of $l$. We let $W_{\alpha}$ to be the $W^i_{\alpha}$ which satisfies both (1) and (2).
 
Next we will see that the statement (3) can be achieved by choosing a 
specialization over a $k_{\alpha}$-point of $S$ which is not a $k'$-point. 
By shrinking $S$ to an open subset, if necessary, we may assume that the covers
$X_F\to F$, $Y_F\to F$ and $W_{XF}\to X_F$ also 
descend to covers of $S$-schemes such that the fibers over every closed point 
in $S$ are covers of smooth $k$-curves with the same Galois group and 
ramification properties as the fiber over the generic point $\spec(k((t)))$ 
of the original covers. Also observe that $F^g=\PP^1_x\otimes_k k((t))$ so 
the specialization of $X_{F,S}\to F_S$ at any point of $S$ is the morphism 
$\Phi_X:X\to \PP^1_x$ and similarly the specialization of $W_{XF,S}\to X_{F,S}$ 
at any point of $S$ is $W_X\to X$.  

Note that the morphism $W^g\to \PP^1_x\otimes_k k((t))$ of $k((t))$-schemes 
factors through $\lambda:Y_F\otimes_{k[[t]]}k((t))\to \PP^1_x\otimes_k k((t))$. 
The latter covering map is locally given by $z^p-z-(x/t)^r$. 
By Artin-Schreier theory for $p$-cyclic extensions, 
the field extension given by irreducible polynomials $z^p - z - ax^r$ 
and $z^p - z - bx^r$ over $k(x)$ for $a,b \in k$ are equal if and only if 
$a/b \in \FF_p$. Since $k_{\alpha}$ has infinite transcendence degree over
$k'$ there exist $a_{\alpha}\in k_{\alpha}$ transcendental over $k'$ 
such that $t=a_{\alpha}$ defines a point $P_{\alpha}$ in $S$ which is not
a $k'$-point. Then $Y_{F,P_{\alpha}}$, the fiber of $Y_F$ over $P_{\alpha}$, is not 
defined over $k'$. Since $T_{P_{\alpha}}\to Y_{F,P_{\alpha}}$ is a dominant morphism, 
$T_{P_{\alpha}}$ is also not defined over $k'$. Hence $W_X\times_X T_{P_{\alpha}}$ is 
also not defined over $k'$.

Note that $X$ and $W_X$ are defined over $k_0$. So $X$ and $W_X$ can be viewed as a
$k'$-curve, a $k_{\alpha}$-curve or a $k$-curve. We shall do so without making 
base change to various fields to simplify notations. Since $Y_{F,P_{\alpha}} 
\to \PP^1_x$ is a $p$-cyclic cover of $k_{\alpha}$-curves, its pull-backs 
$T_{P_{\alpha}}\to X$ and $W_{XT,P_{\alpha}}\to W_X$ are also $p$-cyclic covers of
$k_{\alpha}$-curves. But (1) tells us $W_{XT}$ and $W_X\times_X T$ are isomorphic 
as covers of $T$. So $W_{XT,P_{\alpha}}\cong W_X\times_X T_{P_{\alpha}}$. Hence
$W_{XT,P_{\alpha}}$ is not defined over $k'$. Moreover, $W_{P_{\alpha}}\to 
W_{XT,P_{\alpha}}$ is an $H$-cover of $k_{\alpha}$-curves and by (2), it is not 
a pull-back of any $H$-cover of $W_X$. Hence by Proposition \ref{pro:fieldofdefn}, 
$W_{P_{\alpha}} \to W_{XT,P_{\alpha}}$ is not defined over $k'$. So by Remark 
\ref{basicproperty} the cover $W_{P_{\alpha}}\to T_{P_{\alpha}}$ is not defined 
over $k'$.

For (4), first we observe that $k(T)$ is a compositum of $L_1=k(X)\otimes_k k((t))$
and $L_2$, where $L_2$ is the function field of a dominating irreducible component of 
$$(Y\times_k \spec(k((t))))\times_{\PP^1_y\times_k\spec(k((t)))} (D\times_k \spec(k((t))))$$
Here the morphism $D\times_k \spec(k((t))) \to \PP^1_y \times_k \spec(k((t)))$
is the composition of $D\times_k \spec(k((t))) \to \PP^1_x \times_k \spec(k((t)))$
with $\PP^1_x \times_k \spec(k((t))) \to \PP^1_y \times_k \spec(k((t)))$ where 
the later morphism is defined in local co-ordinates by sending $y$ to $t/x$.
More explicitly $L_2=k((t))(D)[z]/(z^p-z-(x/t)^r)$. Let $Z$ be the normalization of $D\times_k\spec(k((t)))$ in $L_2$. 

Since $T_{\alpha}$ is a specialization of $T$ to $k$, $k(T_{\alpha})$ is a compositum 
of $k(X)$ and $k(Z_{\alpha})$ where $Z_{\alpha}$ is the corresponding specialization 
of $Z$ to $k$. The morphism $Z \to D\times_k \spec(k((t)))$ factors through
$Y\times_k \spec(k((t)))$. So the genus of $Z$ and hence of $Z_{\alpha}$ is greater 
than the genus of $Y$ which in turn is at least $g$.
Also $Z_{\alpha}\to D$ is unramified over $C$ because $T_{\alpha}$ dominates
$Z_{\alpha}$. The Galois extension $k(Z)/k((t))(D)$ is the compositum of the 
the extension $k((t))(Y)/k((t))(t/x)$ with $k((t))(D)$.
Moreover $k((t))(Y)/k((t))(t/x)$ is $p$-cyclic, so the extension $k(Z)/k((t))(D)$ is either trivial or $p$-cyclic. So the
same is true for $k(Z_{\alpha})/k(D)$. 
If $k(Z_{\alpha})/k(D)$ is trivial then $\Gal(K^{un}/k(Z_{\alpha}))=\pi_1(C) \supset \Pi$
by assumption. If $k(Z_{\alpha})/k(D)$ is a $p$-cyclic extension then $\Gal(K^{un}/k(Z_{\alpha})) 
\supset P_g(C)$ as genus of $Z_{\alpha}$ is greater than $g$.
So either way $\Pi \subset \Gal(K^{un}/k(Z_{\alpha}))$. Also $\Pi$ is clearly
contained in $\pi_1^L=\Gal(K^{un}/k(X))$ (see Figure \ref{toweroffields}), 
hence $\Pi \subset \Gal(K^{un}/k(X)k(Z_{\alpha}))=\Gal(K^{un}/k(T_{\alpha}))$.
\end{proof}

\begin{thm}\label{main}
  Suppose $\Pi\subset P_g(C)$ for some $g\ge 0$.
  Under the above notation
  the following finite split embedding problem has $\card(k)=m$ proper solutions
  $$
  \xymatrix{
    &          &               &\Pi \ar[d] \ar @{-->} [dl]\\
    1\ar[r] & H \ar[r] & \Gamma \ar[r] & G \ar[r]\ar[d] & 1\\
    & & & 1 }
  $$
  Here $H$ is a prime-to-$p$ group and a minimal normal subgroup of $\Gamma$.
\end{thm}

\begin{proof}
First we note that, translating the problem to Galois theory using Figure 
\ref{toweroffields}, our objective is to find $m$ distinct $\Gamma$-extensions 
of $K^b$ contained in $K^{un}$ and containing $M$ (or equivalently $k(W_X)$).
By proposition \ref{many-covers}, for each $k_{\alpha} \in \cF$, there exist 
$W_{\alpha}$ and $T_{\alpha}$ such that 
$\Gal(k_{\alpha}(W_{\alpha})/k_{\alpha}(T_{\alpha}))$ is $\Gamma$. 
Moreover, since $k_{\alpha}$ and $k$ are algebraically closed field, 
$\Gal(k(W_{\alpha})/k(T_{\alpha}))$ is also $\Gamma$. Also $k(W_X)$ 
and $K^b$ are linearly disjoint over $k(X)$ and $k(T_{\alpha})\subset K^{b}$.
So $\Gal(k(W_X)k(T_{\alpha})/k(T_{\alpha}))=\Gal(k(W_X)/k(X))=G$. 

Proposition \ref{many-covers}(2) implies that $K^bk(W_X)$ and $k(W_{\alpha})$ 
are linearly disjoint over $k(W_X)k(T_{\alpha})$. 
So $\Gal(K^{b}k(W_{\alpha})/K^{b}k(W_X))= 
\Gal(k(W_{\alpha})/k(T_{\alpha})k(W_X))=H$. Also $\Gal(K^bk(W_X)/K^b)=G$, so 
comparing degrees, we get that $k(W_{\alpha})$ and $K^b$ are linearly disjoint over $k(T_{\alpha})$ and $\Gal(K^bk(W_{\alpha})/K^b)=\Gamma$. 
To complete the proof, it is enough to show that for 
$k_{\alpha} \subsetneq k_{\beta} \in \cF$, 
$k(W_{\alpha})K^{b} \ne k(W_{\beta})K^{b}$ as subfields of $K^{un}$. 
Let $L_o=k_{\alpha}(T_{\alpha}) k_{\beta}(T_{\beta})$. 
Note that $L_o\subset K^{b}$. 

We will first show that 
$k_{\alpha}(W_{\alpha})L_o \ne k_{\beta}(W_{\beta})L_o$. 
Suppose equality holds, then $k_{\beta}(W_{\beta})k_{\alpha}(T_{\alpha})=k_{\alpha}
(W_{\alpha})k_{\beta}(T_{\beta})$. So the extension $k_{\beta}(W_{\beta})k_{\alpha}
(T_{\alpha})/k_{\beta}(T_{\beta})k_{\alpha}(T_{\alpha})$ is the compositum of the 
extension $k_{\alpha}(W_{\alpha})/k_{\alpha}(T_{\alpha})$ with $k_{\beta}(T_{\beta})$.
Hence the $\Gamma$-cover from the normalization of $W_{\beta}\times_X T_{\alpha}$ to the 
normalization of $T_{\beta}\times_X T_{\alpha}$ is defined over $k_{\alpha}$ by 
Lemma \ref{lm:fieldofdefn_extn}.
Also $W_{\beta}\to T_{\beta}$ is a $\Gamma$-cover, $T_{\alpha}\to X$ can be viewed 
as a cover of $k_{\alpha}$-curves and $T_{\beta}\to X$ can be viewed as a cover 
of $k_{\beta}$-curves. So Lemma \ref{lm:descend} 
can be used to conclude that the morphism $W_{\beta} \to T_{\beta}$ is defined over $k_{\alpha}$. This contradicts Proposition \ref{many-covers}(3).

Since $k$ and $k_{\beta}$ are algebraically closed fields, $k_{\beta}(W_{\alpha})L_o \ne 
k_{\beta}(W_{\beta})L_o$ implies that
$k(W_{\alpha})L_o \ne k(W_{\beta})L_o$. Moreover,
both $k(W_{\alpha})L_o$ and $k(W_{\beta})L_o$ are $\Gamma$-extensions of $kL_o$ because $kL_o\subset K^b$
and, $K^b$ and $k(W_{\alpha})$ (respectively $k(W_{\beta})$) are linearly disjoint over $k(T_{\alpha})$ (respectively $k(T_{\beta})$).

Next we shall see that $k(W_{\alpha})L_o$ and $k(W_{\beta})L_o$ are linearly disjoint 
over $k(W_X)L_o$. Assume the contrary and let 
$M'=k(W_{\alpha})L_o\cap k(W_{\beta})L_o$. Then $M'$ is a subfield of 
$k(W_{\alpha})L_o$ which contains
$k(W_X)L_o$ properly. Also being an intersection of two Galois extensions of 
$kL_o$, $M'$ is also a Galois extension of $kL_o$. So the
Galois group $\Gal(k(W_{\alpha})L_o/M')$ is a normal subgroup of 
$\Gamma=\Gal(k(W_{\alpha})L_o/kL_o)$ and a proper
subgroup of $H=\Gal(k(W_{\alpha})L_o/k(W_X)L_o)$ contradicting that 
$H$ is a minimal normal subgroup of $\Gamma$.

Finally, it will be shown that the assumption $k(W_{\alpha})K^b=k(W_{\beta})K^b$ 
leads to a contradiction. Let $L_1=k(W_{\alpha})k(W_{\beta})\cap K^b$. 
Since $k(W_{\alpha})K^b/K^b$ is a $\Gamma$-extension and $K^b$, 
$k(W_{\alpha})k(W_{\beta})$ are linearly disjoint over $L_1$,
the extension degree $[k(W_{\alpha})k(W_{\beta}):L_1]=|G||H|$.  
Also note that $[k(W_{\alpha})k(W_{\beta}):kL_o]=|G||H|^2$ because 
$[k(W_X)L_o:kL_o]=|G|$ 
and, $k(W_{\alpha})L_o$ and $k(W_{\beta})L_o$ are linearly disjoint 
$H$-extensions of $k(W_X)L_o$.
We have the following tower of fields where the labels of the arrow 
denote the degree:

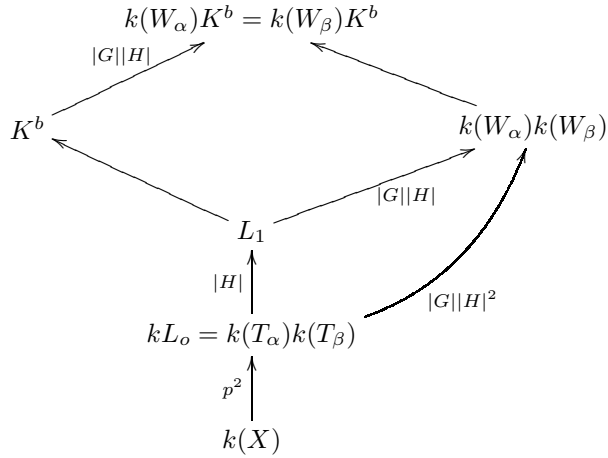
\begin{equation*}
  \xymatrix{ & k(W_{\alpha})K^b=k(W_{\beta})K^b  \\
    K^b\ar[ur]^{|G||H|} &  & k(W_{\alpha})k(W_{\beta}) \ar[ul]\\
    &  L_1\ar[ru]_{|G||H|}\ar[lu]\\
    &  kL_o=k(T_{\alpha})k(T_{\beta}) \ar[u]^{|H|}\ar@/_2pc/[uur]_{|G||H|^2}\\
    &  k(X) \ar[u]^{p^2}
  }
\end{equation*}

From the above figure we also conclude that $[L_1:kL_o]=|H|$. Also observe 
that $[kL_0:k(T_{\alpha})]=p$.
Since $L_1\subset K^b$, $L_1$ and $k(W_{\alpha})$ are linearly disjoint 
over $k(T_{\alpha})$. Also $[L_1:k(T_{\alpha})]=p|H|$ because 
$[kL_o:k(T_{\alpha})]=p$. So we have $[k(W_{\alpha})L_1:k(T_{\alpha})]=p|G||H|^2$. 
Moreover $[k(W_{\alpha})k(W_{\beta}):k(T_{\alpha})]=p|G||H|^2$. So by 
comparing degrees we find that the inclusion $k(W_{\alpha})L_1 \subset 
k(W_{\alpha})k(W_{\beta})$ is in fact an equality $k(W_{\alpha})L_1=k(W_{\alpha})k(W_{\beta})$.

Note that $L_1\subset K^b$ and $K^b/k(X)$ is an abelian extension, so $L_1/k(X)$ is also abelian. 
Since $k(T_{\alpha})k(T_{\beta})/k(X)$ is $(\ZZ/p\ZZ)^2$-extension and $H$ is a 
prime-to-$p$ group, there exist an $H$-Galois extension $L_2/k(X)$ 
such that $L_1=L_2L_o$ and $L_2$, $kL_o$ are linearly disjoint over $k(X)$.

Using Proposition \ref{fieldofdefn_primetop} we see that every 
prime-to-$p$ cover of $X$, \'etale over the inverse image of $C$ is defined
over $k_0$. Hence the extension $L_2/k(X)$ is defined over $k_0\subset k_{\alpha}$. 
So by the Lemma \ref{lm:compositum}, the extension $L_2k(W_{\alpha})/k(T_{\alpha})$ 
is defined over $k_{\alpha}$. But 
$$L_2k(W_{\alpha})k(T_{\beta})=L_2L_0k(W_{\alpha})=L_1k(W_{\alpha})=k(W_{\alpha})k(W_{\beta})$$ 
So by taking compositum with $k(T_{\beta})$ and using Lemma \ref{lm:fieldofdefn_extn}, 
we see that the extension $k(W_{\alpha})k(W_{\beta})/k(T_{\alpha})k(T_{\beta})$ is defined over 
$k_{\alpha}$. Hence by Remark \ref{basicproperty}, the sub-extension
$k(T_{\alpha})k(W_{\beta})/k(T_{\alpha})k(T_{\beta})$ is defined over $k_{\alpha}$.
Also $T_{\alpha}\to X$ is a morphism of $k_{\alpha}$-curves and both 
$W_{\beta}\times_X T_{\alpha}\to T_{\beta}\times_X T_{\alpha}$ and 
$W_{\beta}\to T_{\beta}$ are $\Gamma$-extensions.  So by Lemma \ref{lm:descend}, 
$W_{\beta}\to T_{\beta}$ is defined over $k_{\alpha}$. This again contradicts 
(3) of Proposition \ref{many-covers}.
\end{proof}

\begin{rmk}\label{highgenus}
Note that the assumption $\Pi\subset P_g(C)$ for some $g$ means that the \'etale
pro-cover of $C$ corresponding to the field $K^b$ dominates all $p$-cyclic covers of $C$ of
genus at least $g$. If we relax the above assumption on $\Pi$ by asking that the 
pro-cover corresponding to $K^b$ dominates all 
but a ``small set'' of $p$-cyclic covers of $C$ of genus at least $g$ then also 
Theorem \ref{main} holds with slight modifications in the proof. Here a ``small set''
means a set of cardinality strictly less than $m$.
\end{rmk}

\subsection{Quasi-p group}
Now the embedding problem with quasi-$p$ kernel contained in the $M(\Gamma)$
will be shown to have $m$ distinct solutions.
\begin{pro}\label{p-rankokay}
 Let $\Pi$ be a closed normal subgroup of $\pi_1(C)$ of rank $m$ such that 
 $\pi_1(C)/\Pi$ is abelian then $R_S(\Pi)=m$ for all finite $p$-group $S$.
\end{pro}

\begin{proof}
We observe that the pro-$p$ quotient of $\pi_1(C)$ is isomorphic to the pro-$p$ 
free group of rank $m$ by \cite[Theorem 5.3.4]{ha3}, Lemma \ref{FJ-lemma}
and \cite[Theorem 8.5.2]{RZ}. So the pro-$p$ quotient of $\Pi$ is 
also pro-$p$ free of rank $m$ by \cite[Corollary 8.9.3]{RZ}. 
Hence $R_S(\Pi)=m$ for every finite $p$-group $S$. 
\end{proof}

\begin{thm}\label{quasi-p}
  Suppose $\Pi$ is closed normal subgroup of $\pi_1(C)$ of rank $m$ such that 
  $\pi_1(C)/\Pi$ is abelian.
  Then the following finite embedding problem has $\card(k)=m$ proper solutions
  \begin{equation}\label{EP-quasi}
  \xymatrix{
     &         &               &\Pi \ar[d] \ar @{-->} [dl]\\
    1\ar[r] & H \ar[r] & \Gamma \ar [r] & G \ar[r]\ar[d] & 1\\
      &  & & 1  }
  \end{equation}
  Here $H$ is a quasi-$p$ group, a minimal normal subgroup of $\Gamma$ and it is 
  contained in $M(\Gamma)$. 
\end{thm}

\begin{proof}
Let us recall Figure \ref{toweroffields} and the setup just 
before subsection \ref{sec:prime-to-p}. 
In the notation of Figure \ref{toweroffields},
let $X^0$ be the normalization of $C$ in $L$, 
i.e., $X^0$ is the open subset of $X$ lying above $C$. 
Hence $\pi_1(X^0)$ contains $\Pi$ and the surjection
$\Pi\to G$ extends to $\pi_1(X^0)$. Also note that $\pi_1(X^0)$ is a 
subgroup of $\pi_1(C)$.
By F. Pop's result (\cite{pop},\cite[Theorem 5.3.4]{ha3}), 
the following embedding problem has $m$ distinct solutions  
  $$
  \xymatrix{
             &    &           &\pi_1(X^0) \ar[d] \ar @{-->} [dl]\\
    1\ar[r] &H \ar[r] & \Gamma \ar[r] & G\ar[r]\ar[d] &1 \\
            & & & 1}
  $$

Let $I$ be an indexing set of cardinality $m$ and $\tilde{\theta}_i, i\in I$ denote 
the $m$ distinct solutions to the above embedding problem. 
Let $\theta_i$ be the restriction of $\tilde{\theta}_i$ to the normal subgroup $\Pi$.
By Lemma \ref{restriction-lemma} we know that every 
$\theta_i$ is a solution to the embedding problem \eqref{EP-quasi}.

Now we shall show that the embedding problem \eqref{EP-quasi} has $m$ 
distinct solutions.
For each solution $\tilde \theta_i$ of the above embedding problem for 
$\pi_1(X^0)$, let $L_i$ be the fixed subfield of $K^{un}$ by the group 
$\ker{\tilde \theta_i}$. Note that $\Gal(L_i/L)=\Gamma$. Moreover, since $\theta_i$ 
is a solution to the embedding problem \eqref{EP-quasi}, we have 
$\Gal(L_iK^b/K^b)=\Gamma$. Let us summarize the situation in the following
diagram.

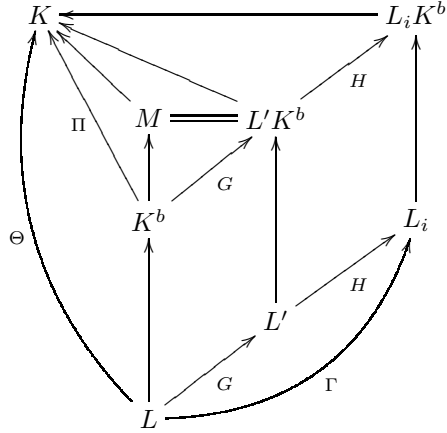
\begin{equation*}
  \xymatrix{K & & & L_iK^b \ar[lll]\\
    &  M\ar@{=}[r]\ar[ul] & L'K^b \ar[ull]\ar[ur]_H\\
    &  K^b\ar[uul]^{\Pi}\ar[u] \ar[ru]_G & & L_i\ar[uu]\\
    &  & L'\ar[ru]_H\ar[uu]\\
    &  L \ar@/^2pc/[uuuul]^{\Theta} \ar[uu] \ar[ru]_G \ar@/_2pc/[uurr]_{\Gamma}
  }
\end{equation*}

Note that $L_i$ and $L'K^b$ are linearly disjoint over $L'$ for all $i\in I$.
Let $i$  and $j$ be two distinct elements of $I$. 
Since $L_i$ and $L_j$ are Galois extensions of $L$, so is $L_i\cap L_j$. Also 
$L'\subset L_i\cap L_j$, so the Galois group $\Gal(L_i/L_i\cap L_j)$
is a normal subgroup of $\Gamma$ and is contained in $H=\Gal(L_i/L')$. 
But $H$ is a minimal normal subgroup of
$\Gamma$, so either $L_i=L_i\cap L_j$ or $L_i\cap L_j=L'$. 
Since $\tilde \theta_i$ and $\tilde \theta_j$ are surjections 
with different kernels, $L_i\ne L_j$. So we must have $L_i\cap L_j=L'$. In 
particular, $\Gal(L_iL_j/L')=H\times H$ and $\Gal(L_iL_j/L)=\Gamma\times_G\Gamma 
=\{(\gamma_1,\gamma_2)| \gamma_1,\gamma_2 \in \Gamma \text{ and }\bar \gamma_1=\bar \gamma_2 \in G\}$. 
Since $H$ is a minimal normal subgroup of 
$\Gamma$, $H=\mathbb{S}\times \mathbb{S}\times\ldots \mathbb{S}$ for some
simple group $\mathbb{S}$. 

If $\mathbb{S}$ is not abelian then $H$ and
hence $H\times H$ is perfect (i.e. it has no non-trivial quotient group
that is abelian). But $L'K^b/L'$ is an abelian extension since it is
a base change of the abelian extension $K^b/L$. Hence $L'K^b$ and
$L_iL_j$ are linearly disjoint over $L'$. Therefore 
$[L_iL_jK^b/K^b]=|G||H|^2$. In particular, $L_iK^b\ne L_jK^b$ and hence
$\theta_i$ for $i\in I$ are all distinct.

Now suppose $\mathbb{S}$ is abelian. Since $H$ is a quasi-$p$ group 
$S=\ZZ/p\ZZ$. Hence $H$ is a $p$-group.
If $L_iK^b\ne L_jK^b$ for every $i, j \in I$, $i\ne j$ then clearly $\theta_i$
for $i\in I$ are all distinct solutions to the embedding problem 
\eqref{EP-quasi} and we are done again.

So suppose for some $i,j\in I$, $i\ne j$, $L_iK^b=L_jK^b$. Let 
$L_1=K^b\cap L_iL_j$.
Since $L_1$ is an intersection of Galois extensions of $L$, $L_1/L$ is a
Galois extension. And we have the following tower of fields.

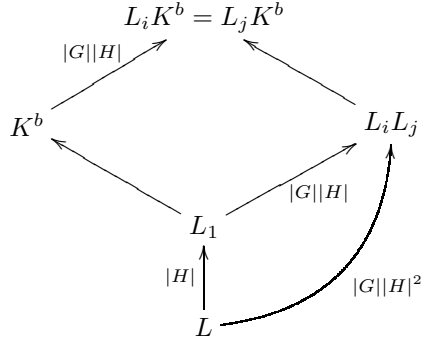
\begin{equation*}
  \xymatrix{ & L_iK^b=L_jK^b  \\
    K^b\ar[ur]^{|G||H|} &  & L_iL_j \ar[ul]\\
    &  L_1\ar[ru]_{|G||H|}\ar[lu]\\
    &  L \ar[u]^{|H|}\ar@/_2pc/[uur]_{|G||H|^2}
  }
\end{equation*}

The labels above denote the degree of the extensions. Note that $K^b$ and 
$L_iL_j$ are linearly disjoint over $L_1$ so 
$[L_iL_j:L_1]=[L_iK^b:K^b]=|\Gamma|=|G||H|$. Also 
$[L_iL_j: L]=|\Gamma\times_G\Gamma|=|G||H|^2$. Hence $[L_1:L]=|H|$. We have
already observed that $L_i$ and $K^b$ are linearly disjoint over $L$ and 
$L_1\subset K^b$, hence $L_i$ and $L_1$ are linearly disjoint over $L$. So
$[L_iL_1:L]=|G||H|^2=[L_iL_j:L]$. Hence the inclusion $L_iL_1 \subset L_iL_j$
is in fact equality. Linear disjointness
of $L_i$ and $L_1$ also tells us that $\Gal(L_iL_1/L)=\Gamma\oplus H_1$ 
where $H_1=\Gal(L_1/L)$. So we have $\Gamma\oplus H_1\cong \Gamma\times_G\Gamma$.
Also note that $|H_1|=|H|$ hence $H_1$ is a $p$-group.

Fix an $i\in I$ and $l$ a positive integer. Let $I'$ be another 
indexing set of cardinality $m$. Using Propostion \ref{p-rankokay} we have 
$R_{H_1^l}(\Pi)=m$. So there are $m$ distinct Galois extensions $M'_{\alpha}$ of 
$K^b$ indexed by $\alpha \in I'$ such that $\Gal(M'_{\alpha}/K^b)=H_1^l$ and $M'_{\alpha}\subset K^{un}$. Since $L_iK^b/K^b$ is finite extension, we can have only finitely many intermediate fields extensions. We choose $l$ to be greater than this
number. Taking fixed subfield of $M'_{\alpha}$ by various copies of $H_1^{l-1}$, 
we get $l$ distinct $H_1$ extensions of $K^b$ such that any one of them is 
linearly disjoint with the compositum of the remaining ones over $K^b$. By choice of 
$l$, one of these extensions must be linearly disjoint with $L_iK^b$ over $K^b$. We will
denote this field by $M_{\alpha}$. 
So for each $\alpha \in I'$ there exist $M_{\alpha}\subset M'_{\alpha}$ such 
that $L_iK^b$ and $M_{\alpha}$ are linearly disjoint over $K^b$ and $\Gal(M_{\alpha}/K^b)=H_1$.
Using the linear disjointness, we see that the Galois group 
$\Gal(M_{\alpha}L_i/K^b)=\Gamma\oplus H_1\cong \Gamma\times_G\Gamma$. 

Let $L^{\alpha}\subset M_{\alpha}L_i$ be  
a $\Gamma$-extension of $K^b$ different from $L_i$ and containing $L'$.
Then $L^{\alpha}$ provides a solution to the embedding problem 
\eqref{EP-quasi}. Again using the fact that $H$ is a minimal normal subgroup 
of $\Gamma$, we observe that $L_iK^b$ and $L_{\alpha}$ are linearly disjoint over
$L'K^b$. Therefore, $[L^{\alpha}L_i:K^b]=|G||H|^2=[M_{\alpha}L_i:K^b]$ and hence
$L^{\alpha}L_i=M_{\alpha}L_i$. 

For $\alpha,\beta \in I'$, we say $\alpha \sim \beta$, if 
$M_{\alpha}L_i=M_{\beta}L_i$. This is clearly an equivalence relation. 
Moreover each equivalence class is finite since $M_{\alpha}L_i$ can 
have only finitely many subfields which are $H_1$-extensions of $K^b$. 
So only finitely many $M_{\beta}$'s are contained in $M_{\alpha}L_i$.
Finally if $\alpha$ and $\beta$ are in two different equivalence classes
then $L^{\alpha}L_i$ and $L^{\beta}L_i$ are distinct, which implies 
$L^{\alpha}$ and $L^{\beta}$ are distinct. Since there are $m$ 
distinct equivalence classes, we obtain $m$ distinct solutions to the embedding
problem \eqref{EP-quasi}. 
\end{proof}

\begin{rmk}
Note that the hypothesis $\Pi\subset P_g(C)$ is not necessary in the above result.
\end{rmk}

\begin{proof}{\bf (of theorem \ref{groupisfree})}
The \'etale fundamental group $\pi_1(C)$ is projective so $\Pi$, 
being a closed subgroup of $\pi_1(C)$, is also projective
(\cite[Proposition 22.4.7]{FJ}). The result now follows from
Theorem \ref{profinite}, Theorem \ref{main} and Theorem \ref{quasi-p}.
\end{proof}

Let $C$ be a smooth affine curve as above. Recall that $\pi_1^c(C)$ is the 
commutator subgroup of $\pi_1(C)$.
\begin{pro}\label{rank}
Let $S$ be a finite simple group. Then 
$R_S(\pi_1^c(C))=m$ and $R_S(P_g(C))=m$ for all $g\ge 0$.
\end{pro}

\begin{proof}
Note that $\pi_1^c(C)$ and $P_g(C)$ for all $g\ge 0$ are closed normal subgroups 
of $\pi_1(C)$ of rank $m$ and $\pi_1^c(C)$ is
contained in $P_0(C)$. Moreover, $\pi_1(C)/\pi_1^c(C)$ and $\pi_1(C)/P_g(C)$ are 
pro-abelian groups. Let $S$ be a finite simple group.
If $S$ is a prime-to-$p$ group then the result follows from Theorem 
\ref{main} by taking 
$H=\Gamma=S$. If $p$ divides $|S|$ then $S$ is a quasi-$p$ simple group. Moreover 
if $S$ is also non-abelian group then the result follows from \cite[Theorem 5.3]{kum}. 
Finally if $S$ is an abelian simple quasi-$p$ group then $S\cong \ZZ/p\ZZ$. Since
$\pi_1(C)/\pi_1^c(C)$ and $\pi_1(C)/P_g(C)$ are abelian groups, the result follows 
from Proposition \ref{p-rankokay}.
\end{proof}

\begin{cor}\label{cor:freenessof-comm-P_g}
The commutator subgroup $\pi_1^c(C)$ of $\pi_1(C)$ is a profinite free group 
of rank $m$ for any smooth affine curve
$C$ over an algebraically closed field $k$ of characteristic $p$ and cardinality
$m$. The subgroups $P_g(C)$ of $\pi_1(C)$ are also free profinite group of rank 
$m$ for all $g\ge 0$.
\end{cor}

\begin{proof}
As observed earlier, $\pi_1^c(C)$ is a closed normal subgroup of $\pi_1(C)$ of 
rank $m$ contained in $P_0(C)$. Also $\pi_1(C)/\pi_1^c(C)$ and $\pi_1(C)/P_g(C)$ are 
abelian groups. So the result follows from Theorem \ref{groupisfree} and Proposition \ref{rank}.
\end{proof} 

The restriction that $\Pi\subset P_g(C)$ for some $g$ can not be dropped 
completely as the following example suggests. 
Though it could be somewhat relaxed (see Remark \ref{highgenus}).
\begin{example} \label{nonfree-ex}
 Let $C$ be the affine line and $\Pi=\cap \{\pi_1(Z)| Z \to  C$ an \'etale cover 
 and $Z$ is again the affine line\}. Clearly $\Pi$
 is a closed normal subgroup of $\pi_1(C)$ and $\pi_1(C)/\Pi$ is an infinite  
 abelian pro-$p$ subgroup. But $\Pi$ has no non-trivial prime-to-$p$ quotients.
 To see this, assume there is one. Then there exists a prime-to-$p$ finite field 
 extension $M/K^b$ with $M\subset K^{un}$. Using finiteness of this 
 field extension, one could get a prime-to-$p$ extension of $L$ where $L\subset K^b$
 is a finite extension of $k(C)$. But the normalization of $C$ in $L$ is also
 an affine line, so it can not have a prime-to-$p$ \'etale cover by Theorem
 \ref{abhconj}.
\end{example}

Let $K_{p^n}$ denote the intersection of all open normal subgroups of 
$\pi_1(C)$ so that the quotient is an abelian group of exponent at most $p^n$.
Let $G_{p^n}=\pi_1(C)/K_{p^n}$ then $G_{p^n}=\invlim \Gal (k(Z)/k(C))$ where
$Z\to C$ is a Galois \'etale cover of $C$ with Galois group $(\ZZ/p^n\ZZ)^l$
for some $l\ge 1$. The group $G_{p^n}$ has a description in terms of Witt rings of the 
coordinate ring of $C$. In fact 
$G_{p^n}\cong \Hom(W_n(\cO_C)/P(W_n(\cO_C)), \ZZ/p^n\ZZ)$ by \cite[Lemma 3.3]{kum}. 
Here $W_n(\cO_C)$ is the ring of Witt vectors of length $n$ and $P$ is a group
homomorphism from $W_n(\cO_C)$ to itself
given by ``Frobenius - Identity" (see Section 2 of \cite{kum} for details).
Hence for any $n\ge 1$, we get the following exact sequence:
$$1\to K_{p^n}\to \pi_1(C)\to \Hom(W_n(\cO_C)/P(W_n(\cO_C)), \ZZ/p^n\ZZ)\to 1$$

\begin{cor}
 $K_{p^n}$ is a profinite free group of rank $m$
\end{cor}

\begin{proof}
Note that $K_{p^n}$ is a closed normal subgroup of $\pi_1(C)$ and it is contained in 
$P_0(C)$. The quotient $\pi_1(C)/K_{p^n}$ is clearly an abelian group. 
Proposition \ref{rank} is also true when $\pi_1^c(C)$ is replaced by 
$K_{p^n}$ and the proof is the same. Hence $K_{p^n}$ is also profinite 
free of rank $m$ in view of Theorem \ref{groupisfree}.
\end{proof}

\end{document}